\documentclass[a4paper,11pt]{article}
\usepackage{amsfonts,amssymb,amsthm,cite,amsmath,amstext}
\usepackage{color}
\usepackage[colorlinks,linkcolor=red,citecolor=blue]{hyperref}
\usepackage{enumerate}
\numberwithin{equation}{section}

\usepackage{amsmath,amssymb,amsthm}
\usepackage{amsmath,amssymb,amsthm,amscd}
\usepackage{color}
\usepackage[mathscr]{eucal}

\title{\textbf{Teissier's problem on proportionality of nef and big classes over a compact K\"ahler manifold}}
\author{\textsc{Jixiang Fu and Jian Xiao}}
\date{}
\begin{document}
\maketitle

\theoremstyle{definition}
\newtheorem*{pf}{Proof}
\newtheorem{theorem}{Theorem}[section]
\newtheorem{remark}{Remark}[section]
\newtheorem{problem}{Problem}[section]
\newtheorem{conjecture}{Conjecture}[section]
\newtheorem{lemma}{Lemma}[section]
\newtheorem{corollary}{Corollary}[section]
\newtheorem{definition}{Definition}[section]
\newtheorem{proposition}{Proposition}[section]
\newtheorem{example}{Example}[section]

\begin{abstract}
We solve Teissier's proportionality problem for transcendental nef classes over a compact K\"ahler manifold which says that the equality in the Khovanskii-Teissier inequalities hold for two nef and big classes if and only if the two classes are proportional. This result recovers the previous one of Boucksom-Favre-Jonsson for the case of nef and big line bundles over a (complex) projective algebraic manifold.
\end{abstract}

\section{Introduction}
Around the year 1979, inspired by the Aleksandrov-Fenchel inequalities in convex geometry, Khovanskii and Teissier discovered independently deep inequalities in algebraic geometry which now is called Khovanskii-Teissier inequalities. These inequalities present a nice relationship between the theory of mixed volumes and algebraic geometry. Their proofs are based on the usual Hodge-Riemmann bilinear relations. A natural problem is how to characterize the equality case in these inequalities for two nef and big line bundles, which was considered by Teissier \cite{Tei82, Tei88}.

In their nice paper \cite{BFJ09}, Boucksom, Favre and Jonsson solved this problem and the answer is that the equality holds if and only if two line bundles are (numerically) proportional. In their paper, they proved an algebro-geometric version of the Diskant inequality in convex geometry following the same strategy of Diskant which is based on the differentiability of the volume function of convex bodies. To obtain their Diskant inequality, they develop an algebraic construction of the positive intersecton products of pseudo-effective classes and use them to prove that the volume function on the N\'{e}ron-Severi space of a projective variety is $\mathcal{C} ^1$-differentiable, expressing its differential as a positive intersection product. Note that their results hold on any complete algebraic variety over an algebraically closed field of characteristic zero. Later, Cutkosky \cite{Cut13} extended these remarkable results to a complete variety over an arbitrary field.

On the other hand, Dinh and Nguy\^en generalized the Hodge-Riemann bilinear relations (and some other results) to compact K\"ahler manifolds in the mixed situation. Using these relations, one can easily get the Khovanskii-Teissier inequalities for transcendental nef classes. So a natural question is how to characterize the equality case in this situation. In this note, we give the same answer of this question as in the algebro-geometric case.

In \cite{BFJ09} and \cite{Cut13}, a key ingredient, in the proof of the differentiability theorem of the volume of big line bundles over a projective variety, and thus in the proof  of the algebro-geometric version of the Diskant inequality, is the weak holomorphic Morse inequality
$$\textup{vol}(A-B)\geq A^n -nA^{n-1}B$$
for any nef line bundles $A$ and $B$. Hence, if one would like to use their methods to extend their results to transcendental classes, the main missing part
is the weak transcendentally holomorphic Morse inequality. However, up to now,
it is not fully proved yet (see \cite{ Xia13, Pop14}). In this note, without using the transcendental version of Diskant inequality, we can still solve Teissier's proportionality problem for transcendental classes. Thus, our result covers the previous one of Boucksom-Favre-Jonsson. Indeed, the key idea in the proof of our main result has been hidden in our previous work \cite{FX14}. For readers' convenience, we will present it in details.

\section{The main theorem}
Let us first recall the definition of nefness and bigness for $(1,1)$-classes on a compact K\"ahler manifold.
Assume $X$ is an $n$-dimensional compact K\"ahler manifold with a K\"ahler metric $\omega$. Let $\alpha\in H^{1,1}_{BC}(X, \mathbb{R})$ be a $(1,1)$ Bott-Chern class. Then $\alpha$ is called {\sl nef} if for any $\varepsilon>0$, there exists a smooth representation $\alpha_\varepsilon \in \alpha$ such that $\alpha_\varepsilon>-\varepsilon\omega$. This definition is equivalent to say that $\alpha$ belongs to the closure of the K\"ahler cone of $X$ which is denoted as $\overline{\mathcal{K}}$. And $\alpha$ is called {\sl big} if there exist a positive number $\delta$ and a positive current $T \in \alpha$ such that $T>\delta\omega$ (such a current $T$ is called a {\sl K\"ahler current}). This is equivalent to say that $\alpha$ belongs to the interior of pseudo-effective cone which is denoted as $\mathcal{E}^{\circ}$. For more notions, such as the movable cone $\overline{\mathcal{M}}$ in the following theorem, one can consult \cite{BDPP13}.

\begin{theorem}
 Assume $X$ is an $n$-dimensional compact K\"ahler manifold. Let $\alpha, \beta\in \overline{\mathcal{K}}\cap \mathcal{E}^{\circ}$ be two nef and big classes. Denote $s_k: = \alpha^k \cdot \beta^{n-k}$. Then the following statements are equivalent:
\begin{enumerate}[(1)]
\item $s_k ^2 =s_{k-1}\cdot s_{k+1}$ \ for $1\leq k\leq n-1$;
\item $s_k ^n =s_{0} ^{n-k}\cdot s_{n} ^{k}$\ for $0\leq k\leq n$;
\item $s_{n-1} ^n =s_{0} \cdot s_{n} ^{n-1}$;
\item  $\textup{vol}(\alpha+\beta)^{1/n} =\textup{vol}(\alpha)^{1/n}+\textup{vol}(\beta)^{1/n}$;
\item  $\alpha$ and $\beta$ are proportional;
\item $\alpha^{n-1}$ and $\beta^{n-1}$ are proportional.
\end{enumerate}
Moreover, all of the above statements are equivalent to
the $(n-1)$-th map
$\alpha\mapsto\alpha^{n-1}$ embedding
the nef big cone $\overline{\mathcal{K}}\cap \mathcal{E}^{\circ}$ into the movable cone $\overline{\mathcal{M}}$.
\end{theorem}

\begin{proof}
For a projective algebraic manifold, the usual Khovanskii-Teissier inequalities imply
\begin{equation}\label{03}
s_k ^2 \geq s_{k-1}\cdot s_{k+1},\quad \textup{for}\ 1\leq k\leq n-1
\end{equation}
if $\alpha$ and $\beta$ are two nef divisors. We remark that it also holds if $\alpha$ and $\beta$ are two transcendental nef classes on a compact K\"ahler manifold (see \cite{DN06}). Its proof is a consequence of Ninh and Nguy\^en's result on mixed Hodge-Riemann bilinear relations for compact K\"ahler manifolds. For example, one can consult Proposition 6.2.1 in \cite{Cao13}. In fact, let $\omega_1$, $\cdots$, $\omega_{n-2}$ be $(n-2)$ K\"ahler classes of $X$. Consider the following quadratic $Q$ on $H^{1,1}_{BC}(X, \mathbb{R})$:
$$
Q(\lambda, \mu):= \int_X \lambda \wedge \mu \wedge \omega_1\wedge...\wedge \omega_{n-2}.
$$
According to \cite{DN06}, $Q$ is of signature $(1, h^{1,1})$. For any $\alpha$, $\beta\in \overline{\mathcal{K}}$ and $t\in \mathbb{R}$, consider $Q(\alpha+t\beta, \alpha+t\beta)$, $t\in\mathbb R$. As a function of $t$,  we claim that $Q(\alpha+t\beta, \alpha+t\beta)$ has at least a real solution. We only need to consider the case when $\alpha$ and $\beta$ are linearly independent and thus, $\alpha$ and  $\beta$ span a 2-dimensional subspace. In view of the signature of $Q$, it can not be positive on this 2-dimensional subspace. Now our claim  follows from this easily. The existence of real solution is equivalent to
$$
\Bigl(\int_X \alpha \wedge \beta \wedge \omega_1\wedge...\wedge \omega_{n-2}\Bigr)^2
\geq \Bigl( \int_X \alpha^2 \wedge \omega_1\wedge...\wedge \omega_{n-2}\Bigr)
\cdot \Bigl(\int_X \beta^2 \wedge \omega_1\wedge...\wedge \omega_{n-2}\Bigr)
.$$
Since $\omega_1, ..., \omega_{n-2}$ are arbitrary, taking appropriate $\omega_i$'s and then taking limits, we obtain the inequalities (\ref{03}) for any two transcendental nef classes.

\vskip5pt

We commence to prove the main result. It is easy to see the equivalences  of (1)-(4) (see e.g. \cite{Cut13}). Since $\alpha$, $\beta\in \overline{\mathcal{K}}\cap \mathcal{E}^{\circ}$, it is clear that $s_k >0$ for $0\leq k \leq n$. We first prove $(1)\Leftrightarrow (3)$. It is trivial that  (1) implies (3). On the other hand, the equalities in (\ref{03}) imply
\begin{equation}\label{02}
\frac{s_{n-1}}{s_0}=\frac{s_{n-1}}{s_{n-2}}\cdot\frac{s_{n-2}}{s_{n-3}}
\cdot...\cdot\frac{s_{1}}{s_{0}}\geq\frac{s_{n-1}}{s_{n-2}}\cdot\frac{s_{n-2}}{s_{n-3}}
\cdot...\cdot\bigl(\frac{s_{2}}{s_{1}}\bigr)^2\geq \cdots \geq\bigl(\frac{s_{n}}{s_{n-1}}\bigr)^{n-1}.
\end{equation}
Thus if (3) holds, then all inequalities in (\ref{02}) are equalities, and hence (1) holds. Now let us prove $(1)\Leftrightarrow (2)$. This also follows from the equalities (\ref{03}), since we will have
\begin{equation}\label{04}
(\frac{s_k}{s_{k-1}})^{n-k} \cdot...\cdot (\frac{s_1}{s_{0}})^{n-k}\geq (\frac{s_n}{s_{n-1}})^{k} \cdot...\cdot (\frac{s_{k+1}}{s_{k}})^{k},
\end{equation}
 which clearly implies $(1)\Leftrightarrow (2)$. Next we prove $(2)\Leftrightarrow (4)$. Inequality (\ref{04}) can be rewritten as
$$s_k ^n\geq s_0 ^{n-k}\cdot s_n ^{k},\quad \textup{for}\ \ 0\leq k\leq n\ .$$
 These inequalities yield
\begin{equation*}
\begin{aligned}
\textup{vol}(\alpha+\beta)=&(\alpha+\beta)^n =\sum \frac{n!}{k!(n-k)!}s_k \\
\geq&
\sum \frac{n!}{k!(n-k)!}s_0 ^{n-k/n}\cdot s_n ^{k/n}\\
=&\bigl(\textup{vol}(\alpha)^{1/n}+\textup{vol}(\beta)^{1/n}\bigr)^n.
\end{aligned}
\end{equation*}
This implies $(2)\Leftrightarrow (4)$.

\vskip5pt

The implication $(5)\Rightarrow (3)$ is trivial. Now the {\sl real} difficulty is to prove  $(3)\Rightarrow (5)$, but which can be finished following from the ideas in our previous work \cite{FX14}. Without loss of generality, assume $\textup{vol}(\alpha)=\textup{vol}(\beta)$.
If (3) holds, we will construct two equal positive $(1,1)$-currents in $\alpha$ and $\beta$ respectively. Hence this implies (5). To prove this, we first construct two positive $(1,1)$-currents in $\alpha$ and $\beta$ respectively which are equal on a Zariski open set. The construction heavily depends on the main theorem in \cite{BEGZ10} which solves Monge-Amp\`ere equations in big cohomology classes. Then, by the support theorem of currents,  the difference of these two currents can only be a combination of some prime divisors. By showing that all the coefficients in the combination vanish, we deduce that these two currents are equal.
 All is all, the key elements in the proof of $(3)\Rightarrow (5)$ are to solve Monge-Amp\`{e}re equations in nef and big cohomology classes and to use  some basic facts in pluripotential theory. In the following, we will carry out all the details.

We will use the same symbol $\alpha$ (resp. $\beta$) to denote a smooth representation in the cohomology class $\alpha$ (resp. $\beta$). Fix a K\"ahler metric $\omega$ and a smooth volume form $\Phi$ with $\int_X \Phi=1$. Since $\alpha$ and $\beta$ are nef and big, Theorem C of \cite{BEGZ10} implies that we can solve the following two degenerate complex Monge-Amp\`ere equations:
\begin{align}
\label{eq degen MA}
\bigl\langle(\alpha+i\partial\bar{\partial}\varphi)^{n}\bigr\rangle
&=c_{\alpha,0}\Phi,\\
\bigl\langle(\beta+i\partial\bar{\partial}\psi)^{n}\bigr\rangle
&=c_{\beta,0}\Phi,
\end{align}
where $\langle\cdot\rangle$
denotes the non-pluripolar products of positive currents, and $c_{\alpha,0}=\text{vol}(\alpha)=\textup{vol}(\beta)=c_{\beta,0}$. Moreover,
$\varphi$ (resp. $\psi$) has minimal singularities and is smooth on the ample locus
$\textup{Amp}(\alpha)$ (resp. $\textup{Amp}(\beta)$), which is a Zariski open set of $X$ depending
only on the cohomology class of $\alpha$ (resp. $\beta$). Let us first briefly recall how the solutions $\varphi$ and $\psi$ are obtained. Indeed, based on Yau's seminal work \cite{Yau78} on the Calabi conjecture, the above two degenerate complex Monge-Amp\`ere equations can be solved by approximation. By Yau's theorem, for $0<t<1$, we can solve the following two families of Monge-Amp\`{e}re equations:
\begin{align}
\label{eq family MA}
(\alpha+t\omega+i\partial\bar{\partial}\varphi_t)^n
& =c_{\alpha,t}\Phi,\\
(\beta+t\omega+i\partial\bar{\partial}\psi_t)^n
& =c_{\beta,t}\Phi,
\end{align}
where $c_{\alpha,t}= \int_X(\alpha+t\omega)^n$, $c_{\beta,t}= \int_X(\beta+t\omega)^n$ and $\sup_X\varphi_t=\sup_X\psi_t =0$.
Denote
$$\alpha_t =\alpha+t\omega+i\partial\bar{\partial}\varphi_t\quad \textup{and}\quad \beta_t =\beta+t\omega+i\partial\bar{\partial}\psi_t\ .$$
 We consider the limits of $\alpha_t$ and  $\beta_t $ as $t$ tends to zero. By basic properties of plurisubharmonic functions,
since $\sup_X\varphi_t=\sup_X \psi_t =0$, the
family of solutions $\{\varphi_t\}$ (resp. $\{\psi_t\}$) is compact in $L^1(X)$-topology. Thus there exists a convergent subsequence which we still denote it by the same symbol $\{\varphi_t\}$ (resp. $\{\psi_t\}$) and  there exists an $\alpha$-psh function $\varphi$ (resp. a $\beta$-psh function $\psi$) such that, when $t$ tends to zero, we have the following limits \emph{in the sense of currents on $X$}:
\begin{equation}\label{eq limit current 1}
\alpha_t\rightarrow
\alpha+i\partial\bar{\partial}\varphi,
\end{equation}
and
\begin{equation}\label{eq limit current2}
\beta_t\rightarrow
\beta+i\partial\bar{\partial}\psi.
\end{equation}
Moreover, by the theory developed in \cite{BEGZ10} and  basic estimates
in \cite{Yau78}, $\varphi_t$ (resp. $\psi_t$) is compact in
$C^{\infty}_{\textup{loc}}(\textup{Amp}(\alpha))$ (resp. $C^{\infty}_{\textup{loc}}(\textup{Amp}(\beta))$). Therefore there exist
convergent subsequences such that
the convergences (\ref{eq limit current 1}) and (\ref{eq limit current2})
 is \emph{in the topology of
$C^{\infty}_{\textup{loc}}(\textup{Amp}(\alpha))$ and $C^{\infty}_{\textup{loc}}(\textup{Amp}(\beta))$}.
Hence $\varphi$ (resp. $\psi$) is smooth on $\textup{Amp}(\alpha)$ (resp. $\textup{Amp}(\beta)$) respectively.
Moreover, since $\Phi$ is a
smooth volume form, $\alpha+i\partial\bar{\partial}\varphi$ (resp. $\beta+i\partial\bar{\partial}\psi$) must be a K\"ahler
metric on $\textup{Amp}(\alpha)$ (resp. $\textup{Amp}(\beta)$).

Denote the Zariski open set $\textup{Amp}(\alpha)\cap \textup{Amp}(\beta)$ by $\textup{Amp}(\alpha, \beta)$, and denote $\alpha+i\partial\bar{\partial}\varphi$ (resp. $\beta+i\partial\bar{\partial}\psi$) by $\alpha_0$ (resp. $\beta_0$). We first claim that $\alpha_0=\beta_0$ on $\textup{Amp}(\alpha, \beta)$.
Let
$c_{t}=c_{\alpha,t}/c_{\beta,t}$. By our assumption $\textup{vol}(\alpha)=\text{vol}(\beta)$, it is clear that
\begin{align}
\label{eq lim c_t}
\underset{t\rightarrow 0}{\textup{lim}}{c_t}=1\ .
\end{align}
Assume $\alpha^{n-1}=\beta^{n-1}+\Theta (\alpha,\beta)$ for some smooth $(n-1,n-1)$-form $\Theta (\alpha,\beta)$. Then
\begin{align}
\label{eq n-1 power}
\alpha_t ^{n-1}=\beta_t ^{n-1} +\Theta_t
\end{align}
for some smooth $(n-1,n-1)$-form $\Theta_t$. Pointwisely,  $\alpha_t$, $\beta_t$, $\alpha_t ^{n-1}$, $\beta_t ^{n-1}$ and $\Theta_t$ can be viewed as matrixes. In this sense, we have
\begin{align}
\label{eq n-1 power 1}
\frac{\textup{det}\ \alpha_t ^{n-1}}{\textup{det}\ \beta_t ^{n-1}}=
\Bigl(\frac{\textup{det}\ \alpha_t }{\textup{det}\ \beta_t}\Bigr)^{n-1}.
\end{align}
Hence, we have
\begin{equation}\label{01}
\begin{aligned}
c_t ^{n-1/n}&=\Bigl(\frac{\textup{det} \alpha_t ^{n-1}}{\textup{det}  \beta_t ^{n-1}}\Bigr)^{1/n}=
\Bigl(\frac{\textup{det} (\beta_t ^{n-1} +\Theta_t)}
{\textup{det} \beta_t ^{n-1}}\Bigr)^{n-1}\\
&\leq 1+ \frac{1}{n}\sum (\beta_t ^{n-1})^{i\bar j}(\Theta_t)_{i\bar j},
\end{aligned}
\end{equation}
where the matrix $\bigl((\beta_t ^{n-1})^{i\bar j}\bigr)_{n\times n}$ is the inverse of $\beta_t ^{n-1}$. Equivalently, multiplying both sides of (\ref{01}) by $\beta_t ^n$, we have
\begin{align}
\label{eq n-1 power ag1}
c_t ^{n-1/n}\beta_t ^n\leq \beta_t ^n+\beta_t \wedge \Theta_t.
\end{align}
Note that
$\beta_t \wedge \Theta_t=\alpha_t ^{n-1}\wedge \beta_t-\beta_t ^n$. Consider $\{\alpha_t ^{n-1}\wedge \beta_t\}$ (resp. $\{ \beta_t ^n\}$) as a family of positive measures, then it is of bounded mass. Thus there exist convergent subsequences, which we still denote as $\alpha_t^{n-1}\wedge \beta_t$ and $\beta_t^n$, and positive measures $\mu_1$ and $\mu_2$ such that
\begin{align}
\label{eq measure convergence}
&\alpha_t ^{n-1}\wedge \beta_t \rightarrow \mu_1,\\
&\beta_t ^n\rightarrow \mu_2
\end{align}
in the sense of measures. If denote $\mu=\mu_1-\mu_2$, then
$\beta_t \wedge \Theta_t\rightarrow \mu$.
We claim that $\mu$ is a zero measure. It is not hard to see from (\ref{eq lim c_t}) and
(\ref{eq n-1 power ag1}) that $\mu$ is a positive measure on
$X$: Let $f$ be any positive smooth function over $X$, we have
$$
\int_X f \mu =\underset{t\rightarrow 0}{\textup{lim}}\int_X f (\beta_t \wedge \Theta_t)
\geq \underset{t\rightarrow 0}{\liminf}\int_X f (c_t ^{\frac{n-1}{n}} \beta_t ^n -\beta_t ^n)
=0.
$$
Meanwhile, the assumption (3) implies
$$
    \int_{X}\mu
    =\lim_{t \to 0}\int_X
    (\beta_{t }\wedge\alpha_{t}^{n-1}-\beta_{t}^{n})
    =\int_X (\beta\wedge\alpha^{n-1}-\beta^{n})=0.
$$
Hence $\mu $ must vanish identically. In particular, since
$\textup{Amp}(\alpha, \beta)$ is a Zariski open set (thus a Borel measurable set),
we have
\begin{align}
\label{app lim measure}
  \beta_{t}\wedge\Theta_{t}\rightarrow 0
\end{align}
in the sense of measures on $\textup{Amp}(\alpha, \beta)$.
Using the convergence (\ref{eq limit current 1}) and (\ref{eq limit current2}) in the topology of
$C^{\infty}_{\textup{loc}}(\textup{Amp}(\alpha))$ and $C^{\infty}_{\textup{loc}}(\textup{Amp}(\beta))$, it is clear that there exists some smooth form $\Theta_{0}$ which is only defined on $\textup{Amp}(\alpha,\beta)$ such that $\Theta_{t}\to\Theta_0$ in the topology of $C^\infty_{\textup{loc}}(\textup{Amp}(\alpha,\beta))$. This implies in the same topology
\begin{align}
\label{smooth lim1}
   \beta_{t}\wedge\Theta_{t}\to\beta_{0}\wedge\Theta_{0}\ .
\end{align}
Combining (\ref{app lim measure}) and (\ref{smooth lim1}), and
using uniqueness of the limit, we obtain
\begin{align}\label{eq zero}
    \beta_{0}\wedge\Theta_{0}=0
\end{align}
on $\textup{Amp}(\alpha,\beta)$.
The above equality (\ref{eq zero}) implies that if we take the limits on $\textup{Amp}(\alpha,\beta)$ of both sides of (\ref{01}), we have
\begin{align}
 \label{arith-geom ineq zero}
   1&=\bigl(\frac{\det\alpha_{0}^{n-1}}{\det\beta_{0}^{n-1}}\bigr)^{\frac{1}{n}}=
   \bigl(\frac{\det(\beta_{0}^{n-1}+\Theta_{0})}{\det\beta_{0}^{n-1}}\bigr)^{\frac{1}{n}} \\
 &\leq 1+\frac{1}{n}
 \sum_{i,j}(\beta_{0}^{n-1})^{i\bar{j}}(\Theta_{0})_{i\bar{j}}\\
 &=1.
\end{align}
This forces $\Theta_0=0$, and hence
 $\alpha_0^{n-1}=\beta_0^{n-1}$ on $\textup{Amp}(\alpha,\beta)$. Since $\alpha_0$ and $\beta_0$ are K\"ahler metrics,
we have $\alpha_0=\beta_0$ on $\textup{Amp}(\alpha,\beta)$.
We claim $\alpha_0=\beta_0$ on {\sl all} $X$.
Before going on, we need the following two lemmas.

\begin{lemma}\label{current spt} (see \cite{Dem}, pp. 142-143)
Let $T$ be a $d$-closed $(p,p)$-current. Suppose $\textup{supp}\ T$ is  contained in an analytic subset $A$.
If $\dim A<n-p$, then $T=0$; if $T$ is of order zero and $A$ is of pure dimension $n-p$ with $(n-p)$-dimensional
irreducible components $A_{1},\cdots,A_{k}$, then $T=\sum c_{j}[A_{j}]$ with $c_{j}\in \mathbb{C}$.
\end{lemma}

\begin{lemma}\label{lelong num zero}
(see \cite{Bou04}, Proposition 3.2 and Proposition 3.6)
Let $\alpha$ be a nef and big class, and let $T_{\min}$ be a positive current in $\alpha$ with minimal singularities.
Then the Lelong number $\nu(T_{\min},x)=0$ for any point $x\in X$.
\end{lemma}

It is clear that $S:= X\setminus \textup{Amp}(\alpha, \beta)$ is a proper analytic subset of $X$. Let $T=\alpha_0-\beta_0$, then $T$ is a real $d$-closed (1,1)-current and $\textup{supp}\ T\subset S$.
If codim$S\geq 2$, then $T=0$ according to Lemma \ref{current spt}. This implies $\alpha_{0}=\beta_{0}$ on $X$; If
codim$S=1$ and $S$ has only irreducible components $D_{1},\cdots,D_{k}$ of pure dimension one,
then Lemma \ref{current spt} implies $\alpha_{0}-\beta_{0}=\sum c_{j}[D_{j}]$; If codim$S=1$ and $S$ has also components of codimension more than one, we just repeat the proof of Lemma \ref{current spt} in \cite{Dem} (for more details, see \cite{FX14}), and still get $\alpha_{0}-\beta_{0}=\sum c_{j}[D_{j}]$.
Since $\alpha_0$ and $\beta_0$ are real, all
$c_j$ can be chosen to be real numbers.
If there exists at least one $c_j>0$, we write this equality as
\begin{align}
\label{current eq}
    \alpha_{0}-\sum c_{j'}[D_{j'}]=\beta_{0}+\sum c_{j''}[D_{j''}]
\end{align}
with $c_{j'}\leq 0$ and $c_{j''}>0$.
Fix a $j''$ which we denote as $j_0''$. We take a generic point $x\in D_{j''_0}$, for example,
we can take such a point $x$ with $\nu([D_{j''_0}],x)=1$ and $x\notin \underset{j\neq j_0''}{\cup}D_{j}$.
Then taking the Lelong number at the point $x$ on  both sides of (\ref{current eq}), we have
\begin{equation*}\label{lelong eq}
    \nu(\alpha_{0},x)-\sum c_{j'}\nu([D_{j'}],x)=\nu(\beta_{0},x)+\sum c_{j''}\nu([D_{j''}],x).
\end{equation*}
Since $\alpha_{0}$ and $\beta_{0}$ are positive currents with minimal singularities
in nef and big classes, Lemma \ref{lelong num zero} tells us that $\nu(\alpha_{0},x)=0$ and $\nu(\beta_0,x)=0$.
The property of $x$ also implies $\nu([D_{j'}],x)=0$ and $\nu([D_{j''}],x)=0$ for all $j'$ and all $j''\not=j''_0$.
All these force $c_{j_0''}=0$,
which contradicts to our assumption  $c_{j_0''}>0$.  Thus we have
\begin{equation*}
    \alpha_{0}-\sum c_{j'}[D_{j'}]=\beta_{0}.
\end{equation*}
By the same reason, we can also prove $c_{j'}=0$. Hence we finish the proof of $\alpha_0=\beta_0$ over $X$ and of the implication $(3)\Rightarrow (5)$.

\vskip5pt

The implication $(5)\Rightarrow(6)$ is trivial, and it is clear $(5)\Rightarrow(3)$. For the implication of $(6)\Rightarrow (3)$, suppose $\alpha^{n-1}=c \beta^{n-1}$ for some $c>0$, then we have $\alpha^n =c \beta^{n-1}\cdot \alpha \geq c (\beta^n)^{n-1/n}(\alpha^n)^{1/n} $, and 
$\alpha^{n-1}\cdot \beta=c \beta^n \geq (\alpha^n)^{n-1/n}(\beta^n)^{1/n}$. 
This yields $(\alpha^n)^{n-1/n}=c (\beta^n)^{n-1/n}$, and as a consequence, we get $\alpha^{n-1}\cdot \beta=c \beta^n=(\alpha^n)^{n-1/n}(\beta^n)^{1/n}$ which is just $(3)$. Summarizing all the above arguments, we have finished the proof of the equivalences of $(1)-(6)$.

Moreover, the statement that the $(n-1)$-th exterior power map
$\alpha\mapsto\alpha^{n-1}$ embeds
the nef big cone $\overline{\mathcal{K}}\cap \mathcal{E}^{\circ}$ into the movable cone $\overline{\mathcal{M}}$ is equivalent to the implication $(6)\Rightarrow (5)$, and hence to the equivalence $(6)\Leftrightarrow(5)$. Therefore we finish the proof of our theorem.
\end{proof}

\noindent
\textbf{Acknowledgements}: Xiao would like to thank China Scholarship Council and Institut Fourier for support. He also would like to thank Professor Jean-Pierre Demailly for encouragement.
Fu is supported in part by NFSC 11121101.

\noindent
\textsc{Jixiang Fu}\\
\textsc{Institute of Mathematics, Fudan University, 200433 Shanghai, China} \\
\verb"Email: majxfu@fudan.edu.cn"\\

\noindent
\textsc{Jian Xiao}\\
\textsc{Institut Fourier, Universit\'{e} Joseph Fourier-Grenoble I, 38402 Saint-Martin d'H\`{e}res, France} \\
\textsc{and}\\
\textsc{Institute of Mathematics, Fudan University, 200433 Shanghai, China}\\
\verb"Email: jxiao10@fudan.edu.cn"\\


\begin{thebibliography}{99}

\bibitem[Bou04]{Bou04}
S. Boucksom, {\sl Divisorial Zariski decompositions on compact complex manifolds}, Ann. Sci. \'{E}cole Norm. Sup. {\bf 37} (2004), no. 1, 45-76.

\bibitem[BDPP13]{BDPP13}
S. Boucksom, J.-P. Demailly, M. Paun, T. Peternell, {\sl The pseudo-effective cone of a compact K\"ahler manifold and varieties of negative Kodaira dimension}, J. Algebraic Geom. {\bf 22} (2013) 201-248.

\bibitem[BEGZ10]{BEGZ10}
S. Boucksom, P. Eyssidieux, V. Guedj, A. Zeriahi,
{\sl Monge-Amp\`ere equationas in big cohomology classes}, Acta Math.
{\bf 205} (2010), 199-262.



\bibitem[BFJ09]{BFJ09}
S. Boucksom, C. Favre, M. Jonsson, {\sl Differentiability of volumes of divisors and a problem of Teissier}, J. Alg. Geom. {\bf 18} (2009), 279-308.

\bibitem[Cao13]{Cao13}
J. Cao. {\sl Th\'{e}or\`{e}mes d'annulation et th\'{e}or\`{e}mes de structure sur les vari\'{e}t\`{e}s k\"ahleriennes compactes}.
PhD Thesis, Grenoble, 2013.

\bibitem[Cut13]{Cut13}
S. D. Cutkosky, {\sl Teissier's problem on inequalities of nef divisors over an arbitrary field}, arXiv: 1304. 1218v1.

\bibitem[Dem]{Dem}
J.-P. Demailly,  {\sl Complex Analytic and Differential Geometry},
http://www-fourier.ujf-grenoble.fr/~demailly/books.html.

\bibitem[DN06]{DN06}
T.-C. Dinh, V.-A. Nguy\^{e}n, {\sl The mixed Hodge-Riemann bilinear relations for compact K\"ahler manifolds}, Geom. Funct. Anal. {\bf 16} (2006) 838-849.

\bibitem[FX14]{FX14}
J.-X. Fu, J. Xiao, {\sl Relations between the K\"ahler cone and the balanced cone of a K\"ahler manifold}, Adv. Math. {\bf 263} (2014) 230-252.

\bibitem[Pop14]{Pop14}
D. Popovici, {\sl An observation relative to a paper by J. Xiao}, arXiv:1405.2518v1 [math.DG].

\bibitem[Tei82]{Tei82}
B. Teissier, {\sl Bonnesen-type inequalities in algebraic geometry. I. Introduction
to the problem}, In {\sl Seminar on Differential Geometry}, pp. 85-105. Ann. Math.
Stud., 102. Princeton University Press, 1982.

\bibitem[Tei88]{Tei88}
B. Teissier, {\sl Mon\^{o}mes, volumes et multiplicit\'{e}s}. In {\sl Introduction \`{a} la th\'{e}orie
des singularit\'{e}s II}, pp.127-141. Travaux en Cours, 37. Hermann, Paris, 1988.

\bibitem[Xia13]{Xia13}
J. Xiao, {\sl Weak transcendental holomorphic Morse inequalities on compact K\"ahler manifolds}, arXiv preprint 2013, to appear in Annales de l'Institut Fourier.

\bibitem[Yau78]{Yau78}
S.-T. Yau, {\sl On the Ricci curvature of a compact
K\"ahler manifold and the complex Monge-Amp\`{e}re equation}, I, Comm.
Pure. Appl. Math. {\bf 31} (1978), 339-411.
\end{thebibliography}
\end{document}